\documentclass[11pt]{article}

\usepackage[T1]{fontenc}
\usepackage[latin9]{inputenc}
\usepackage{geometry}
\usepackage{amsmath}
\usepackage{amsthm}
\usepackage{amssymb}
\usepackage{thmtools}
\usepackage{thm-restate}
\usepackage{bm}
\usepackage{color}
\usepackage{paralist}
\usepackage{csquotes}
\usepackage[unicode=true,pdfusetitle,
 bookmarks=true,bookmarksnumbered=false,bookmarksopen=false,
 breaklinks=false,pdfborder={0 0 0},pdfborderstyle={},backref=false,colorlinks=false]
 {hyperref}
 
\usepackage[square,sort,comma,numbers]{natbib} 

\usepackage{setspace}
\makeatletter

\usepackage[nameinlink,capitalise,noabbrev]{cleveref}

\hypersetup{%
    bookmarksnumbered, bookmarksopen=true, bookmarksopenlevel=1,%
}

\theoremstyle{plain}
\newtheorem{thm}{Theorem}[section]
\crefname{thm}{Theorem}{Theorems}
\theoremstyle{plain}
\newtheorem{lem}[thm]{Lemma}
\newtheorem{obs}[thm]{Observation}
\crefname{lem}{Lemma}{Lemmas}
\theoremstyle{plain}

\theoremstyle{plain}
\newtheorem*{claim*}{Claim}
\crefname{claim}{Claim}{Claims}
\theoremstyle{definition}

\theoremstyle{plain}

\newtheorem{ques}[thm]{Question}

\theoremstyle{definition}

\crefformat{equation}{#2(#1)#3}
\crefname{appsec}{Appendix}{Appendices}

\DeclareMathOperator{\ssat}{ssat}
\DeclareMathOperator{\sat}{sat}
\DeclareMathOperator{\Prob}{\mathbb{P}}

\usepackage{appendix}

\crefformat{equation}{#2(#1)#3}

\let\originalleft\left
\let\originalright\right
\renewcommand{\left}{\mathopen{}\mathclose\bgroup\originalleft}
\renewcommand{\right}{\aftergroup\egroup\originalright}
\usepackage{verbatim}

\makeatletter
\renewcommand*{\UrlTildeSpecial}{%
  \do\~{%
    \mbox{%
      \fontfamily{ptm}\selectfont
      \textasciitilde
    }%
  }%
}%
\let\Url@force@Tilde\UrlTildeSpecial
\makeatother

\makeatother
\begin{document}

\title{Two problems in graph Ramsey theory}

\author{Tuan Tran\thanks{Discrete Mathematics Group, Institute for Basic Science (IBS), Daejeon, Republic of Korea. Email:
\href{mailto:tuantran@ibs.re.kr} {\nolinkurl{tuantran@ibs.re.kr}}.
This work was supported by the Institute for Basic Science (IBS-R029-Y1).}}

\date{}

\maketitle
\global\long\def\calC{\mathcal{C}}
\global\long\def\F{\mathbb{F}}
\global\long\def\RR{\mathbb{R}}
\global\long\def\cE{\mathcal{E}}
\global\long\def\E{\mathbb{E}}
\global\long\def\N{\mathbb{N}}
\global\long\def\I{\mathcal{I}}
\global\long\def\one{\boldsymbol{1}}
\global\long\def\range#1{\left[#1\right]}
\global\long\def\cP{\mathcal{P}}
\global\long\def\cL{\mathcal{L}}
\global\long\def\calG{\mathcal{G}}
\global\long\def\Pr{\mathrm{P}}

\begin{abstract}
	
We study two problems in graph Ramsey theory. 
In the early 1970's, Erd\H{o}s and O'Neil considered a generalization of Ramsey numbers. Given integers $n,k,s$ and $t$ with $n \ge k \ge s,t \ge 2$, they asked for the least integer $N=f_k(n,s,t)$ such that in any red-blue coloring of the $k$-subsets of $\{1, 2,\ldots, N\}$, there is a set of size $n$ such that either each of its $s$-subsets is contained in some red $k$-subset, or each of its $t$-subsets is contained in some blue $k$-subset. Erd\H{o}s and O'Neil found an exact formula for $f_k(n,s,t)$ when $k\ge s+t-1$. In the arguably more interesting case where $k=s+t-2$, they showed $2^{-\binom{k}{2}}n<\log f_k(n,s,t)<2n$ for sufficiently large $n$. Our main result closes the gap between these lower and upper bounds, determining the logarithm of $f_{s+t-2}(n,s,t)$ up to a multiplicative factor.

Recently,  Dam\'asdi, Keszegh, Malec, Tompkins, Wang and Zamora initiated the investigation of saturation problems in Ramsey theory, wherein one seeks to minimize $n$ such that there exists an $r$-edge-coloring of $K_n$ for which any extension of this to an $r$-edge-coloring of $K_{n+1}$ would create a new monochromatic copy of $K_k$. We obtain essentially sharp bounds for this problem.
\end{abstract}

\section{Introduction}

The Ramsey number $R(n)$ is the smallest natural number $N$ such that every two-coloring of the edges of $K_N$ contains a monochromatic clique of size $n$. The existence of these numbers is guaranteed by Ramsey's theorem \cite{Ramsey30}. Classic results of Erd\H{o}s and Szekeres \cite{ES35} and Erd\H{o}s \cite{Erdos47} imply ${2^{n/2}<R(n)<4^n}$ for every $n\ge 3$. While there have been several improvements on these bounds (see \cite{Colon09,Sah20,Spencer75}), the constant factors in the exponents have remained unchanged for over seventy years.
Given these difficulties, it is natural that the field has stretched in different directions.  One such direction is to try to generalize Ramsey's theorem.

\subsection{A generalization of Ramsey numbers}
\label{sec:shadow Ramsey}
In the early 1970's, Erd\H{o}s and O'Neil \cite{EO73} considered the following generalization of Ramsey numbers. Given integers $n,k,s$ and $t$ with $n \ge k \ge s,t \ge 2$, define $N=f_k(n,s,t)$ to be the minimum integer with the property that, in any red-blue coloring of the $k$-subsets of an $N$-element set, there exists a subset of size $n$ for which either each of its $s$-subsets is contained in some red $k$-subset, or each of its $t$-subsets is contained in some blue $k$-subset. For example, $f_2(n,2,2)$ is the Ramsey number $R(n)$. It is worth noting that a variant of $f_k(n,2,2)$ has been studied recently in \cite{STWZ19,GMOV20}.

Erd\H{o}s and O'Neil \cite{EO73} showed $f_k(n,s,t)=2n-s-t+1$ when $k=s+t-1$. The next interesting case occurs for $k=s+t-2$, where they proved $2^{2^{-\binom{k}{2}}n}<f_k(n,s,t)<4^n$ assuming $n$ is sufficiently large. They claimed that the upper bound can be improved further to $f_k(n,s,t)<2^{o_k(1)n}$. This result is interesting as it indicates the separation between $f_{s+t-2}(n,s,t)$ and $R(n)$. 

The aforementioned lower and upper bounds for $f_{s+t-2}(n,s,t)$ are far apart if we allow $s$ and $t$ to grow with $n$. In this paper, we close this gap, determining the logarithm of $f_{s+t-2}(n,s,t)$ asymptotically.

\begin{thm}\label{thm:ramsey}
For all sufficiently large $n$ and all $s,t$ with $2 \le s \le t \le \log n/(120\log \log n)$, we have
	\[
	2^{\frac{n}{32t}\log\big(\frac{2et}{s}\big)}\le f_{s+t-2}(n,s,t) \le 2^{\frac{4n}{t}\log\big(\frac{2et}{s}\big)}.
	\]
\end{thm}

One can reduce finding $f_{s+t-2}(n,s,t)$ to a graph Ramsey problem. Let $g(n,s,t)$ be the smallest integer $N$ such that in any $N$-vertex graph $G$ one can find an $n$-element vertex set which does not contain both a size-$s$ clique and a size-$t$ independent set. The key point (see \cref{lem:connection}) is 
\begin{equation*}
f_{s+t-2}(n,s,t)=g(n,s,t) \enskip \text{for} \enskip n \ge s+t-2.
\end{equation*}
\cref{thm:ramsey} then follows from the following bounds on $g(n,s,t)$, which may be of independent interest.

\begin{restatable}{prop}{ramsey}\label{prop:ramsey}
Suppose that $n$ is sufficiently large and $2\le s\le t \le \log n/(120\log \log n)$. Then
\[
2^{\frac{n}{32t}\log(\frac{2et}{s})} \le g(n,s,t) \le 2^{\frac{4n}{t}\log(\frac{2et}{s})}.
\]	
\end{restatable}

\subsection{Semisaturated Ramsey numbers}
Given an integer $r\ge 2$, let $\calG_r$ denote the family of complete graphs whose edges are colored with $r$ colors $1,2,\ldots,r$. For $G,G'\in \calG_r$, we say $G'$ extends $G$ if $G'$ can be obtained from $G$ by iteratively adding a new vertex and colored edges connecting the new vertex with the existing ones. A member $G$ of $\calG_r$ is called $(r,K_{k})$-{\em semisaturated} if every $G'\in \calG_r$ that extends $G$ must contain a new monochromatic $K_{k}$. 

We are interested in the smallest size an $(r,K_k)$-semisaturated graph can have, that is,
\[
\ssat_r(K_k):=\min\big \{|V(G)|\colon \text{$G\in \calG_r$ is $(r,K_{k})$-semisaturated}\big \}.
\]
From the definition, it is clear that $\ssat_2(K_k)< R(k)$, and more generally, $\ssat_r(K_k)$ is less than the $r$-color Ramsey number of $K_k$.
Dam\'asdi et al. \cite{DKMTWZ20} initiated the study of $\ssat_r(K_k)$, establishing the following result.

\begin{thm}[Dam\'asdi, Keszegh, Malec, Tompkins, Wang, Zamora]\label{thm:saturated}
\textcolor{white}{ }
\begin{itemize}
	\item[\rm (i)] $(r-1)k^2-(3r-4)k+(2r-3)\le \ssat_r(K_k)$, and the equality holds when $r=2$.
	\item[\rm (ii)] $\ssat_r(K_k) \le (k-1)^r$.
	\item[\rm (iii)] $\ssat_r(K_k) \le 48k^2r^{k^2}$.
\end{itemize}
\end{thm}

The upper bound from (ii) is polynomial in $k$ for fixed $r$, while the upper bound from (iii) is polynomial in $r$ for fixed $k$. The second main result of this paper gives rather sharp estimates for $\ssat_r(K_k)$ in these regimes.

\begin{restatable}{thm}{fixed}\label{thm:saturated fixed k or r}
\textcolor{white}{ }
\begin{itemize}
\item[\rm (i)] For fixed $r\ge 2$, one has
\[
\ssat_r(K_k)=\Theta(k^2).
\]
\item[\rm (ii)] For all $k\ge 3$, there exists a constant $C=C(k)>0$ such that for all $r \ge 3$, 
\[
\tfrac14 r^2 \le \ssat_r(K_k) \le C(\log r)^{8(k-1)^2}r^2.
\]
\end{itemize}
\end{restatable}

The proof of the upper bound in \cref{thm:saturated fixed k or r} (i) requires $k$ to be large with respect to $r$.  Moreover, the exponent of the $\log r$ factor in the second upper bound depends on the size of the clique.  Therefore, we also prove an upper bound on $\ssat_r(K_k)$ which is polynomial in both parameters.

\begin{restatable}{thm}{general}
\label{thm:saturated general}	
For every $k\ge 2$ and every $r\ge 3$, 
\[
\ssat_r(K_k) \le 8k^3r^3.
\]
\end{restatable}

Our proofs of \cref{thm:saturated fixed k or r,thm:saturated general} exploit connections between $\ssat_r(K_k)$ and the vertex Folkman numbers.
To learn more about the vertex Folkman numbers, we refer the interested readers to \cite{HRS18} and the references therein.

\subsection{Organization and notation}

The paper is organized as follows. In \cref{sec:ramsey} we prove \cref{thm:ramsey}. In \cref{sec:saturated} we justify \cref{thm:saturated fixed k or r,thm:saturated general}. We close this paper with some concluding remarks.

We write $[N]$ for the set $\{1,2,\ldots,N\}$, while $\binom{X}{k}$ is the family of all $k$-element subsets of a set $X$. Given two functions $f$ and $g$ of some underlying parameter $n$, we employ the following asymptotic notation: $f=o(g)$ or $g=\omega(f)$ means that $\lim_{n\rightarrow\infty}f(n)/g(n)=0$, while $f=O(g)$ means that $|f|\le Cg$ for some absolute constant $C>0$. We omit floor and ceiling signs where the argument is unaffected. Finally, all logarithms are to the base $2$. 

\section{Proof of \cref{thm:ramsey}}\label{sec:ramsey}

Recall that $g(n,s,t)$ is the smallest integer $N$ such that in any $N$-vertex graph $G$ one can find an $n$-element vertex set which does not contain both a size-$s$ clique and a size-$t$ independent set. On the other hand, $f_k(n,s,t)$ is the minimum integer $N$ with the property that, in any red-blue coloring of the $k$-subsets of an $N$-element set, there exists a subset of size $n$ for which either each of its $s$-subsets is contained in some red $k$-subset, or each of its $t$-subsets is contained in some blue $k$-subset. Our starting point is the following connection between $f_k(n,s,t)$ and $g(n,s,t)$.

\begin{lem}\label{lem:connection}
Provided $s,t \ge 2$ and $n\ge s+t-2$ we have
\[
f_{s+t-2}(n,s,t)=g(n,s,t).
\]
\end{lem}

\begin{proof}
Throughout the proof,  set $k=s+t-2$.
In order to show $f_k(n,s,t) \le g(n,s,t)$, one needs to justify that in any red-blue coloring $\chi$ of the $k$-subsets of $[g(n,s,t)]$ there exists a set of size $n$ such that either each of its $s$-subsets is contained in some red $k$-subset, or each of its $t$-subsets is contained in some blue $k$-subset.  For this purpose, we define a graph $G_{\chi}$ on the vertex set $[g(n,s,t)]$ as follows. Consider a pair $\{x,y\}$ of vertices.  If $\{x,y\}$ is contained in some $s$-subset $S$ for which all $k$-subsets containing $S$ are blue, then we assign $\{x,y\}$ to $E(G_{\chi})$.  If $\{x,y\}$ is contained in some $t$-subset $T$ such that all $k$-subsets containing $T$ are colored red, then we declare $\{x,y\}\notin E(G_{\chi})$.  Finally, if $\{x,y\}$ satisfies neither of the previous two conditions we arbitrarily decide whether $\{x,y\}$ belongs to $E(G_{\chi})$ or not.   The graph $G_{\chi}$ is well-defined: if $\{x,y\}$ satisfies both conditions, then $|S\cup T| \le s+t-2=k$ and every $k$-subset containing $S\cup T$ receives two colors, which is impossible. 

By the definition of $g(n,s,t)$, one can find an $n$-element vertex set $V$ which does not contain either a size-$s$ clique, or a size-$t$ independent set.  Without loss of generality we can assume that the former case occurs. We claim that every $s$-subset $S$ of $V$ is contained in some red $k$-subset of $[g(n,s,t)]$, implying $f_k(n,s,t) \le g(n,s,t)$.  Indeed, if $S$ is contained in blue $k$-subsets only, then $S$ induces a clique of size $s$ in $G_{\chi}$, a contradiction. 

What's left is to show that $f_k(n,s,t) \ge g(n,s,t)$. To this end, let $N=g(n,s,t)-1$.  According to the definition of $g(n,s,t)$, there exists a graph $G$ on $[N]$ so that every set of $n$ vertices contains both a clique of size $s$ and an independent set of size $t$. Consider a red-blue coloring of the $k$-subsets of $[N]$ defined as follows.  Given a $k$-subset $K$, we color $K$ blue if $G[K]$ contains a clique of size $s$, color $K$ red if $G[K]$ contains an independent set of size $t$,  and color $K$ arbitrarily otherwise.  The coloring is well-defined: a $k$-subset cannot host both a clique of size $s$ and an independent set of size $t$ as they would intersect at two (or more) vertices.  Obviously, a size-$s$ clique is contained in blue $k$-subsets only, and a size-$t$ independent set is contained in red $k$-subsets only.  It thus follows from the assumption on $G$ that every set of $n$ vertices contains an $s$-subset all of whose containing $k$-subsets are blue and a $t$-subset all of whose containing $k$-subsets are red. Hence $f_k(n,s,t)\ge N+1=g(n,s,t)$, completing our proof.
\end{proof}

\cref{thm:ramsey} clearly follows \cref{lem:connection} and \cref{prop:ramsey}. For reader's convenience we restate \cref{prop:ramsey} here.

\ramsey*

The rest of this section is devoted to the proof of \cref{prop:ramsey}. Naturally, the proof is broken into two lemmas corresponding to the upper and lower bounds of $g(n,s,t)$.

\begin{lem}\label{lem:ramsey-UpperBound}
Provided $n$ is sufficiently large and $2\le s\le t \le n/\log n$, we have 
\[
g(n,s,t) \le 2^{\frac{4n}{t}\log (\frac{2et}{s})}.
\]
\end{lem}

\begin{lem}\label{lem:ramsey-LowerBound}
Suppose $n$ is sufficiently large and $2\le s \le t \le \log n/(120\log \log n)$. Then
\[
g(n,s,t) \ge 2^{\frac{n}{32t}\log(\frac{2et}{s})}.
\] 
\end{lem}

We will prove \cref{lem:ramsey-UpperBound} using the Erd\H{o}s-Szekeres bound for the off-diagonal Ramsey numbers.

\begin{proof}[Proof of \cref{lem:ramsey-UpperBound}]
The off-diagonal Ramsey number $R(a,b)$ is the smallest natural number $N$ such that any $N$-vertex graph $G$ contains either a clique of size $a$ or an independent set of size $b$. The Erd\H{o}s-Szekeres bound \cite{ES35} says 
$R(a,b) \le \binom{a+b}{a}$ for any $a,b\ge 1$.
In the case $1 \le a \le b$, this implies $R(a,b) \le \binom{2b}{a}< 2^{a\log(\frac{2eb}{a})}$. Thus, for any integer $m\ge 1$ we have
\[
R\Big(\frac{\log m}{\log(\frac{2et}{s})},\frac{t\log m}{s\log(\frac{2et}{s})}\Big) < 
2^{\frac{\log m}{\log(\frac{2et}{s})}\cdot \log(\frac{2et}{s})}=2^{\log m}=m.
\]
We conclude that in any $m$-vertex graph one can find either a clique of size $\frac{\log m}{\log(\frac{2et}{s})}$ or an independent set of size $\frac{t\log m}{s\log(\frac{2et}{s})}$.

Let $N:=2\cdot 2^{\frac{3n}{t}\log(\frac{2et}{s})}$, $a:=\frac{\log(N/2)}{\log(\frac{2et}{s})}=\frac{3n}{t}$, and $b:=\frac{t\log(N/2)}{s\log(\frac{2et}{s})}=\frac{3n}{s}$. Since $2\le s\le t \le n/\log n$, we have $N\ge 2^{\frac{3n}{t}} \ge 2^{3\log n}=n^3$. Moreover, as $2\le s\le t$, 
we see that $a \le b \le \frac{t\log N}{4}$. Consider any graph $G$ on $N$ vertices. From the above discussion, we know that $G$ must contain either a clique of size $a$ or an independent set of size $b$.
By removing this set and repeating $s+t-3$ times, we obtain either $t-1$ vertex disjoint cliques of the same size $a$ or $s-1$ vertex disjoint independent sets of the same size $b$, because the number of removed vertices is always at most 
\[
(s+t-3)\cdot \max\{a,b\} \le 2t \cdot \frac{t\log N}{4} \le N/2
\]
for $t\le n/\log n$ and $N\ge n^3$.
In the latter case, the union of these sets has at least $(s-1)b=3(s-1)n/s \ge 1.5n$ vertices and does not contain a clique of size $s$.
In the former, the union of those sets has cardinality at least $(t-1)a=3(t-1)n/t \ge 1.5n$ and does not contain an independent set of size $t$. Therefore, $g(n,s,t)\le N=2\cdot 2^{\frac{3n}{t}\log(\frac{2et}{s})}\le  2^{\frac{4n}{t}\log (\frac{2et}{s})}$.
\end{proof}

The proof of \cref{lem:ramsey-LowerBound} uses a container theorem of Balogh and Samotij \cite[Proposition 6.1]{BS19}.

\begin{lem}[Balogh and Samotij]
\label{lem:containers}
	For all sufficiently large $n$ and all $r$ with $3\le r \le \log n/(120 \log \log n)$, there exists a collection $\calC$ of at most $2^{n^{2-1/(8r)}}$ subgraphs of $K_n$ such that:
	\begin{itemize}
		\item[\rm (a)] Each $K_r$-free subgraph of $K_n$ is contained in some member of $\calC$,
		\item[\rm (b)] Each $G\in\calC$ has fewer than $\left(1-\frac{1}{2r}\right)\binom{n}{2}$ edges.
	\end{itemize}
\end{lem}

{\bf Remark.} In \cite{BS19}, the above statement is proved with property (b) replaced by
\begin{itemize}
\item[\rm (b')] Each $G\in \calC$ either has less than $n^2/8$ edges or it contains a subgraph $G'$ with ${e(G')>e(G)-o(n^2)}$ that has fewer than $n^{r-1/2}$ copies of $K_r$.
\end{itemize}
A supersaturation result of Balogh, Bushaw, Collares, Liu, Morris and Sharifzadeh \cite[Theorem 1.2]{BBCLMS17} shows property (b) follows from property (b'). 

We are now in a position to justify \cref{lem:ramsey-LowerBound}.

\begin{proof}[Proof of \cref{lem:ramsey-LowerBound}]
Let $N=2^{\frac{n}{32t}\log (\frac{2et}{s})}$, and $p=\frac{s}{2et}\log(\frac{2et}{s})$. 
According to \cref{lem:containers}, there exists a collection $\calC$ of at most $2^{n^{2-1/(8s)}}$ subgraphs of $K_n$ satisfying:
	\begin{itemize}
		\item[(a)] Each $K_s$-free subgraph of $K_n$ is contained in some member of $\calC$,
		\item[(b)] Each $G\in\calC$ has fewer than $\left(1-\frac{1}{2s}\right)\binom{n}{2}$ edges.
	\end{itemize}
Let $\calG(n,p)$ denote the Erd\H{o}s-R\'enyi random graph with $n$ vertices and edge density $p$. It follows from properties (a) and (b) that
\[
\Prob(\calG(n,p) \enskip \text{is $K_s$-free}) \le \sum_{G\in \calC}(1-p)^{\binom{n}{2}-e(G)}
 \le 2^{n^{2-1/(8s)}} (1-p)^{(1+o(1))n^2/(4s)}.
 \]
By the same argument, we get $\Prob(\calG(n,p) \enskip \text{has no independent sets of size $t$}) \le 2^{n^{2-1/(8t)}} p^{(1+o(1))n^2/(4t)}$. Thus in $\calG(N,p)$ the expected number of vertex sets of size $n$ which does not contain both a clique of size $s$ and an independent set of size $t$ is at most
\begin{align*}
&\binom{N}{n}\cdot 2^{n^{2-1/(8s)}} (1-p)^{(1+o(1))n^2/(4s)}+\binom{N}{n}\cdot 2^{n^{2-1/(8t)}} p^{(1+o(1))n^2/(4t)}\\
& \le \exp_2\Big(n\log N+n^{2-\frac{1}{8s}}-(1+o(1))\frac{pn^2}{4s}\Big) + \exp_2\Big(n\log N+n^{2-\frac{1}{8t}}+(1+o(1))\log(p)\frac{n^2}{4t}\Big)\\
&= \exp_2\Big(\frac{epn^2}{16s}+n^{2-\frac{1}{8s}}-(1+o(1))\frac{pn^2}{4s}\Big) + \exp_2\Big(\log\left(\frac{2et}{s}\right)\frac{n^2}{32t}+n^{2-\frac{1}{8t}}+(1+o(1))\log(p)\frac{n^2}{4t}\Big)\\
&= o(1),
\end{align*}
where in the second line we used the inequalities $\binom{N}{n}\le N^n$ and $1-p \le 2^{-p}$, in the third line we substituted $\log N=\frac{n}{32t}\log(\frac{2et}{s})$ and $p=\frac{s}{2et}\log(\frac{2et}{s})$,
and in the last passage we used the estimate $\log p \le - \frac{1}{4}\log\left(\frac{2et}{s}\right)$.
This completes our proof.
\end{proof}

\section{Proofs of \cref{thm:saturated fixed k or r,thm:saturated general}}\label{sec:saturated}

We will use the following observation several times, often without referring to it explicitly.

\begin{obs}
Let $k\ge 3$ and $s\ge r\ge 2$.	Suppose that $G_1,\ldots,G_s$ are edge-disjoint subgraphs of the complete graph on $V$ such that for every $i\in [s]$ and for every $U\in \binom{V}{|V|/r}$ the graph $G_i[U]$ contains a copy of $K_{k-1}$. Then 
\[
\ssat_r(K_k) \le |V|.
\]
\end{obs}

We prove Part (i) of \cref{thm:saturated fixed k or r} via an explicit construction.	

\fixed*

\begin{proof}[Proof of \cref{thm:saturated fixed k or r} (i)]
From \cref{thm:saturated} (i), we learn that $\ssat_r(K_k)=\Omega_r(k^2)$.
To complete the proof, it remains to show that $\ssat_r(K_k)=O_r(k^2)$, a task we now begin.

Without loss of generality we may assume that $k\ge 6r.$
By Chebyshev's Theorem there exists a prime $q$  with $3rk \le q \le 6rk$.
Let $\F_q^2$ be the affine plane over the $q$-element field $\F_q$, with point
set $\cP$ and line set $\cL$. The common vertex set of our graphs $G_1,\ldots,G_r$ is $V:=\cP$. Note that $n:=|V|=q^2 \le 36r^2k^2$. We partition $\cL$ arbitrarily into $r$ families $\cL_1,\ldots,\cL_r$, each consists of $\frac{q^2+q}{r}$ lines. As the edges of $G_i$ we take exactly those pairs $u,v \in \cP$ which lie in a line from $\cL_i$.

Since there is exactly one line passing through two given points, the graphs $G_i$ form an $r$-edge-coloring of the complete graph on $V$. To finish the proof, it suffices to show that for every $i\in [r]$ and every subset $U\in \binom{V}{|V|/r}$ the graph $G_i[U]$ contains a copy of $K_k$.

For a line $\ell \in \cL_i$, let $p_{\ell}$ be the number of points $p\in U$ such that $p\in \ell$. It follows from a point-line incidence bound (see e.g. \cite{Vinh11}) that
\begin{align*}
\sum_{\ell \in \cL_i}p_{\ell} &\ge \frac{|U||\cL_i|}{q}-2q^{1/2}\sqrt{|U||\cL_i|}\\
&\ge \frac{q^3}{r^2}-2q^{1/2}\sqrt{\frac{q^4}{r^2}}\\
&\ge \frac{q^3}{2r^2},
\end{align*} 
where in the last passage we used the assumptions that $q\ge 3kr$ and  $k\ge 6r$. By the pigeonhole principle, there must be a line $\ell \in \cL_i$ with $p_{\ell} \ge \frac{q^3}{2r^2|\cL_i|} \ge k$. Thus $G_i[U]$ contains a copy of $K_k$ within $\ell$. This completes our proof.
\end{proof}

We next proceed to the proof of Part (ii) of \cref{thm:saturated fixed k or r}. 

\begin{proof}[Proof of \cref{thm:saturated fixed k or r} (ii)]
The upper bound follows swiftly from a result of Fox, Grinshpun, Libenau, Person and Szab\'o \cite{FGLPS16}, which in turn is based heavily on the work of Dudek, Retter and R\"{o}dl \cite{DRR14}. Following \cite{FGLPS16}, we call a sequence of pairwise edge-disjoint graphs $G_1,\dots, G_r$ on the same vertex set $V$ a {\em color pattern} on $V$; we assign color $i$ to all the edges of $G_i$. A color pattern $G_1,\ldots, G_r$ is called $K_k$-{\em free} if none of the $G_i$ contains $K_k$ as a subgraph.  A graph with colored vertices and edges is called {\em strongly monochromatic} if all its vertices and edges have the same color.  An $r$-coloring is a function $\chi\colon V\rightarrow [r]$.

Define $P_r(k-1)$ to be the least integer $n$ such that there exists a $K_k$-free color pattern $G_1,\ldots, G_r$ on an $n$-element vertex set $V$ with the property that any $r$-coloring of $V$ contains a strongly monochromatic $K_{k-1}$.

Inspecting the definitions of $\ssat_r(K_k)$ and $P_r(k-1)$, we see that $\ssat_r(K_k) \le P_r(k-1)$. Moreover, it follows from \cite[Lemma 4.3]{FGLPS16} that $P_r(k-1) \le C(\log r)^{8(k-1)^2}r^2$, where $C=C(k)>0$ is a constant. It is worth pointing out that
\cite[Theorem 1]{GW20} implies $P_r(2)=\Theta(r^2\log r)$.
Therefore, we get the desired bound
\[
\ssat_r(K_k) \le C(\log r)^{8(k-1)^2}r^2.
\]

For the lower bound, we first establish the following recursion

\begin{equation}\label{eq:recursion}
\ssat_r(K_k) \ge \ssat_{r-1}(K_k)+r/2 \quad \forall r \ge 2,k \ge 3.
\end{equation}
Clearly, \eqref{eq:recursion} forces $\ssat_r(K_k) \ge \frac12 \sum_{i=2}^{r} i \ge \frac14 r^2$.

Our remaining task is to justify \eqref{eq:recursion}. Take a smallest $(r,K_k)$-semisaturated graph $G\in \calG_r$. Then $G$ has $n=\ssat_r(K_k)$ vertices. \cref{thm:saturated} (i) implies 
$n\ge (r-1)k^2-(3r-4)k+(2r-3) \ge r$.

For $i\in [r]$, denote by $G_i$ the $i$-th color class of $G$. Some color class, say $G_1$, has at most $\frac{1}{r}\binom{n}{2}$ edges. By Tur\'an's theorem applied to the complement of $G_1$, $G_1$ contains an independent set $I$ of cardinality $|I|\ge \frac{n^2}{n+2e(G_1)} \ge \frac{n^2}{n+n(n-1)/r} \ge r/2$, as $n\ge r$. Since $I$ does not contain $K_{k-1}$ at all, the restriction of $G$ to the vertex set $V(G)\setminus I$ is $(r-1,K_k)$-semisaturated. It follows that
\[
\ssat_r(K_k)=|V(G)\setminus I| + |I| \ge \ssat_{r-1}(K_k)+r/2,
\]
completing the proof.
\end{proof}

We close this section with a proof of \cref{thm:saturated general}. Once again it is done via an explicit construction.

\general*

\begin{proof}
We learn from Chebyshev's Theorem that the interval $[kr, 2kr]$ contains a prime number $q$. 
Let $\F_q$ be the finite field of order $q$. 
The vertex set of our graph will be $V := \F_q^3$. By the choice of $q$, $n:=|V|=q^3 \le 8k^3r^3$. 
For each $\lambda \in \F_q$, we will define an incidence structure 
$\I_{\lambda} = (V,\cL_{\lambda} )$ where $\cL_{\lambda}$ is a family of lines in $ \F_q^3$. 
For $\lambda \in \F_q$, let
\[ C_{\lambda} := \Big\{ (1,\lambda, \mu) : \mu \in \F_q\Big\}. \]
Clearly, $C_{\lambda_1}\cap C_{\lambda_2}=\emptyset$ if $\lambda_1 \ne \lambda_2$.

A line in $\F_q^3$ is a set of the form $\ell_{\bm{s},\bm{v}} = \{ \beta \bm{s} + \bm{v}  :  \beta \in \F_q \}$ for some $\bm{s}\in \F_q^3\setminus\{(0,0,0)\}$ and $\bm{v}\in \F_q^3$, where $\bm{s}$ is called the {\em slope}. 
We define 
\[ \cL_{\lambda} := \{ \ell_{\bm{s},\bm{v}} : \bm{s} \in C_{\lambda}, \bm{v} \in \F_q^3\}. \]
Since each 
line contains exactly $q$ points, $|\cL_{\lambda}| =|C_{\lambda}| \frac{q^3}{q} = q^3$.
Each structure $\I_{\lambda}$ enjoys the following properties: 
\begin{itemize}
	\item[(P1)] Every point $v \in V$ is contained in $q$ lines from $\cL_{\lambda}$;
	\item[(P2)] Any two points lie in at most one line; 
\end{itemize}
Furthermore, for $\lambda_1 \neq \lambda_2$, we have
\begin{itemize}
	\item[(P3)] $\cL_{\lambda_1} \cap \cL_{\lambda_2} = \emptyset$.
\end{itemize}

Each incidence structure $\I_{\lambda} = (V,\cL_{\lambda})$ gives rise to a graph $G_{\lambda}$ with vertex set $V$, and edge set
\[
E(G_{\lambda})=\{xy\colon x,y \in \ell \enskip \text{for some} \enskip \ell \in \cL_{\lambda}\}.
\]
By (P3), the graphs $G_{\lambda}$ form a partial $q$-edge-coloring of the complete graph on $V$. To establish the theorem, it suffices to prove that for every $\lambda \in \F_q$ and every subset $U\in \binom{V}{|V|/r}$ the graph $G_{\lambda}[U]$ contains a copy of $K_k$. Indeed, according to (P1), each point $v\in U$ is contained in $q$ lines from $\cL_{\lambda}$. By the pigeonhole principle, we thus can find a line $\ell \in \cL_{\lambda}$ such that 
\[
|\ell \cap U| \ge
\frac{q|U|}{|\cL_{\lambda}|} \ge \frac{q\cdot (q^3/r)}{q^3}=\frac{q}{r} \ge k.
\]
It follows that $G_{\lambda}[\ell \cap U]$ contains a copy of $K_k$. This completes the proof.
\end{proof}

\section{Concluding remarks}\label{sec:conclusion}
We have obtained good bounds for two Ramsey-type problems by exploring their connections to other well-studied problems. In particular, we have shown that
\[
\tfrac14 r^2 \le \ssat_r(K_k) \le C_k (\log r)^{8(k-1)^2} r^2,
\]
which is tight up to the logarithmic factor. However, it is not clear whether the lower bound or the upper bound is closer to the truth. This naturally leads us to the following modest questions.

\begin{ques}
Is it true that $\ssat_r(K_k)=\omega(r^2)$ as $r\rightarrow \infty$?
\end{ques}

\begin{ques}
Does there exist a constant $C$ (independent of $k$) such that $\ssat_r(K_k)=O_k\left((\log r)^Cr^2\right)$?
\end{ques}

In \cite{DKMTWZ20}, Dam\'asdi et al. also studied the behavior of a sibling of the function $\ssat_r(K_k)$.  Let $\calG_r$ denote the family of complete graphs whose edges are colored with $r$ colors (numbered by $1,2,\ldots,r$). A member $G$ of $\calG_r$ is called $(r,K_{k})$-{\em saturated} if for every $i\in [r]$, the graph $G$ does not contain a monochromatic $K_k$ of color $i$, but every $G'\in \calG_r$ that extends $G$ must contain a monochromatic $K_{k}$. 
Let
\[
\sat_r(K_k):= \min \{ |V(G)|\colon G \in \calG_r \enskip \text{is $(r,K_k)$-saturated}\}.
\]
From the definitions of $\sat_r(K_k)$ and $\ssat_r(K_k)$, we see that $\ssat_r(K_k) \le \sat_r(K_k)$. Unlike $\ssat_r(K_k)$, our understanding of $\sat_r(K_k)$ is still in its infancy state: the only known upper bound, obtained by Dam\'asdi et al. \cite{DKMTWZ20}, says that $\sat_r(K_k)\le (k-1)^r$. The obvious open questions here are the following.

\begin{ques}
Does there exist an absolute constant $C>0$ such that $\sat_r(K_k)=O_r(k^C)$?
\end{ques}

\begin{ques}\label{ques:triangle}
Is there an absolute constant $C>0$ such that $\sat_r(K_k)=O_k(r^C)$?
\end{ques}

The triangle case of \cref{ques:triangle} is of particular interest to us.


\begin{thebibliography}{30}

\bibitem{BBCLMS17}
J.~Balogh, N.~Bushaw, M.~Collares, H.~Liu, R.~Morris, and M.~Sharifzadeh, {\it The typical structure of graphs with no large cliques}, Combinatorica {\bf 37} (2017), 617--632.
 
\bibitem{BS19}
J.~Balogh, and W.~Samotij, {\it An efficient container lemma}, 
Discrete Analysis (2020), Paper No. 17, 56 pp.

\bibitem{Colon09}
D.~Conlon, {\it A new upper bound for diagonal Ramsey numbers}, Ann. Math. {\bf 170} (2009), 941--960.

\bibitem{DKMTWZ20}
G.~Dam\'asdi, B.~Keszegh, D.~Malec, C.~Tompkins, Z.~Wang, and O.~Zamora, {\it Saturation problems in the Ramsey theory of graphs, posets and point sets}, European J. Combin.
{\bf 95}
 (2021),
 103321.

\bibitem{DRR14}
A.~Dudek, T.~Retter, and V.~R\"{o}dl, {\it On generalized Ramsey numbers of Erd\H{o}s and Rogers}, J. Combin. Theory B {\bf 109} (2014), 213--227.

\bibitem{Erdos47} 
P.~Erd\H{o}s, {\it Some remarks on the theory of graphs}, Bull. Amer. Math. Soc. {\bf 53} (1947), 292--294.

\bibitem{EO73}
P.~Erd\H{o}s and P.~O'Neil, {\it On a generalization of Ramsey numbers}, Discrete Math. {\bf 4} (1973), 29--35.

\bibitem{ES35} 
P.~Erd\H{o}s and G.~Szekeres, {\it A combinatorial problem in geometry}, Compos. Math. {\bf 2} (1935), 463--470.

\bibitem{FGLPS16}
J.~Fox, A.~Grinshpun, A.~Liebenau, Y.~Person, and T.~Szab\'o, {\it On the minimum degree of minimal Ramsey graphs for multiple colours},  J. Combin. Theory B {\bf 120} (2016), 64--82.

\bibitem{GMOV20}
D.~Gerbner, A.~Methuku, G.~Omidi, M.~Vizer, {\it Ramsey problems for Berge hypergraphs}, SIAM J. Discrete Math. {\bf 34} (2020), 351--369.

\bibitem{GW20}
H. ~Guo and L.~Warnke, {\it Packing Nearly Optimal Ramsey $R(3,t)$ Graphs},
Combinatorica {\bf 40} (2020),  63--103.

\bibitem{HRS18}
H.~H\`an, V.~R\"{o}dl, and T.~Szab\'o, {\it Vertex Folkman numbers and the minimum degree of minimal Ramsey graphs}, SIAM J. Discrete Math. {\bf 32} (2018), 826--838.

\bibitem{Ramsey30} 
F.~P.~Ramsey, {\it On a problem of formal logic}, Proc. London Math. Soc. {\bf 30} (1930), 264--286.

\bibitem{Sah20}
A.~Sah, {\it Diagonal Ramsey via effective quasirandomness}, submitted. arXiv preprint arXiv:2005.09251.

\bibitem{STWZ19}
N.~Salia, C.~Tompkins, Z.~Wang and O.~Zamora, {\it Ramsey numbers of Berge-hypergraphs and related structures}, Electron. J. Combin. {\bf 26} (2019), P4.40.


\bibitem{Spencer75}
J.~H.~Spencer, {\it Ramsey's theorem -- a new lower bound}, J. Combin. Theory Ser. A {\bf 18} (1975), 108--115.

\bibitem{Vinh11}
L.~A.~Vinh, {\it The Szemer\'edi--Trotter type theorem and the sum-product estimate in finite fields}, European J. Combin. {\bf 32} (2011), 1177--1181.
	
\end{thebibliography}
\end{document}